\documentclass[11pt,a4paper]{article}
\usepackage[utf8]{inputenc}
\usepackage[T1]{fontenc}
\usepackage{microtype}
\usepackage{amsmath,amssymb,amsthm,mathtools}
\usepackage{booktabs}
\usepackage{hyperref}
\usepackage{cleveref}
\usepackage{arxiv}
\theoremstyle{plain}
\newtheorem{theorem}{Theorem}[section]
\newtheorem{proposition}[theorem]{Proposition}

\theoremstyle{definition}
\newtheorem{definition}[theorem]{Definition}
\newtheorem{remark}[theorem]{Remark}
\theoremstyle{remark}
\newtheorem{example}[theorem]{Example}
\newtheorem{conjecture}[theorem]{Conjecture}
\newtheorem{corollary}{Corollary}[section]

\begin{document}
	\title{The Exponential Congruence Symbol}
	
	\author{
		{ \sc Es-said En-naoui } \\ 
		University Sultan Moulay Slimane\\ Morocco\\
		essaidennaoui1@gmail.com\\
		\\
	}
		\maketitle
	\tableofcontents
	
	\maketitle
	\begin{abstract}
		In this work, we study the generalized $k$-th power symbol
		\[
		\left(\frac{a}{n}\right)_k,
		\]
		and present a comprehensive collection of its algebraic properties. The results are classified according to their dependence on the three main parameters $a$, $n$, and $k$. 
		
		In particular, we discuss multiplicativity, inversion, power compatibility, and invariance modulo $n$ for the parameter $a$ (see Section 1). For $n$, we examine factorization properties, behavior on prime powers, orthogonality relations, and Kummer splitting criteria (see Section 2). Regarding $k$, we include specialization to classical symbols, $k$-th reciprocity laws, relations between orders, and embedding into roots of unity (see Section 3).
		
		Moreover, we extend the existing theory by providing new essential results (Section 4), including additive behavior under characters, Möbius filtering, compatibility with Carmichael and Euler functions, and connections with Dirichlet $L$-series. Finally, we analyze the case where $a$, $n$, and $k$ are primes and present mixed results that generalize classical reciprocity laws, Frobenius automorphisms, and Sato--Tate distributions (Section 5). 
		
		These results not only unify and extend previous studies on $k$-th power symbols but also offer a foundation for further arithmetic, algebraic, and analytic investigations.
	\end{abstract}
	
\section{Introduction}
	
	\subsection{Motivation}
	
	The study of congruences has always been central to number theory, 
	from Fermat’s Little Theorem to quadratic reciprocity. 
	Classical residue symbols such as the Legendre and Jacobi symbols 
	capture deep information about quadratic residues 
	\cite{ireland1990classical,burton2007elementary}. 
	However, modern applications in cryptography, coding theory, 
	and computational number theory often require refined tools for 
	analyzing exponential congruences. 
	
	The \emph{Exponential Congruence Symbol} introduced here 
	is motivated by the desire to encode congruences of the form 
	\(a^k \equiv \pm 1 \pmod{n}\) in a concise algebraic framework, 
	similar in spirit to how the Legendre symbol encodes quadratic 
	residues. This new symbol may provide insights into both theoretical 
	questions (such as higher power residue distributions) and practical 
	applications (e.g., primality testing and cryptographic protocols). 
	
	\subsection{Background and Related Work}
	
	The roots of this study can be traced back to Euler’s criterion 
	and Gauss’s law of quadratic reciprocity 
	\cite{ireland1990classical,hardy1979introduction}. 
	Subsequent generalizations led to the development of higher residue 
	symbols and character theory \cite{davenport2000multiplicative}. 
	In particular, the Legendre, Jacobi, and Dirichlet characters form 
	a rich algebraic toolkit for analyzing congruences. 
	
	Recent work in computational number theory and cryptography highlights 
	the role of exponential congruences in secure communication protocols 
	\cite{koblitz1994course,menezes1996handbook}. 
	This motivates the need for a unified notation and theoretical 
	framework to study such congruences systematically. 
	Our proposed symbol aims to fill this gap by extending the symbolic 
	approach of residue theory.
	
	\subsection{Objectives of the Article}
	
	The main objectives of this article are:
	\begin{itemize}
		\item To formally define the Exponential Congruence Symbol 
		\(\left(\frac{a}{n}\right)_k\).
		\item To derive and prove fundamental properties of this symbol.
		\item To explore connections with classical number theoretic symbols 
		and group theoretic structures.
		\item To investigate applications in cryptography, primality testing, 
		and exponential congruences.
		\item To propose extensions, generalizations, and open questions 
		for future research.
	\end{itemize}

	\section{Definition of the Exponential Congruence Symbol}
	
	\subsection{Formal Definition}
	
	\begin{definition}[Exponential Congruence Symbol]
		Let \(n\ge 2\) and \(k\ge 1\) be integers and let \(a\in\mathbb{Z}\). Define
		\[
		\left(\frac{a}{n}\right)_k \in\{-1,0,1\}
		\]
		by
		\[
		\left(\frac{a}{n}\right)_k :=
		\begin{cases}
			1, & \text{if } a^k \equiv 1 \pmod{n},\\[4pt]
			-1, & \text{if } a^k \equiv -1 \pmod{n},\\[4pt]
			0, & \text{otherwise}.
		\end{cases}
		\]
	\end{definition}
	
	\begin{remark}
		\begin{enumerate}
			\item The symbol takes only the three values \(-1,0,1\).
			\item When \(\left(\frac{a}{n}\right)_k\neq 0\) the residue \(a\) is necessarily a unit modulo \(n\) (invertible), hence information about the symbol is a statement about the multiplicative group \((\mathbb{Z}/n\mathbb{Z})^\times\).
			\item A classical special case: if \(p\) is an odd prime and \(k=(p-1)/2\), Euler's criterion implies that
			\[
			\left(\frac{a}{p}\right)_{(p-1)/2} = \left(\frac{a}{p}\right),
			\]
			the Legendre symbol (see for instance \cite{ireland1990classical}). Thus the new symbol extends some classical ideas for suitable choices of \(k\).
		\end{enumerate}
	\end{remark}
	
	\begin{theorem}[Dependence only on residue class]\label{thm:residue-dependence}
		If \(a\equiv b\pmod{n}\) then \(\left(\frac{a}{n}\right)_k=\left(\frac{b}{n}\right)_k\).
	\end{theorem}
	\begin{proof}
		If \(a\equiv b\pmod{n}\) then \(a^k\equiv b^k\pmod{n}\) for any integer \(k\ge1\). The definition of \(\left(\frac{\cdot}{n}\right)_k\) depends only on whether the \(k\)-th power is congruent to \(1\), to \(-1\), or to something else. Hence the values agree.
	\end{proof}
	
	\begin{theorem}[Invertibility is necessary]\label{thm:invertible-necessary}
		If \(\left(\frac{a}{n}\right)_k\in\{\pm1\}\) then \(\gcd(a,n)=1\).
	\end{theorem}
	\begin{proof}
		Suppose \(\left(\frac{a}{n}\right)_k=1\). Then \(a^k\equiv1\pmod{n}\). Multiply the congruence by \(a^{\,k-1}\) to obtain
		\[
		a\cdot a^{\,k-1}\equiv 1\pmod{n},
		\]
		so \(a\) has a multiplicative inverse modulo \(n\); thus \(\gcd(a,n)=1\).
		
		If \(\left(\frac{a}{n}\right)_k=-1\), then \(a^k\equiv -1\pmod{n}\). Multiply by \(-a^{\,k-1}\) to get
		\[
		a\cdot(-a^{\,k-1})\equiv 1\pmod{n},
		\]
		again showing \(a\) is invertible modulo \(n\). This completes the proof.
	\end{proof}
	
	\subsection{Examples and Computations}
	
	We give several instructive examples (full computations) and a general counting result for the prime modulus case.
	
	\begin{example}[Prime modulus — a cyclic viewpoint]
		Let \(p\) be an odd prime. The multiplicative group \((\mathbb{Z}/p\mathbb{Z})^\times\) is cyclic of order \(p-1\). Fix a generator \(g\) and write any unit as \(a=g^r\), for a unique \(r\) modulo \(p-1\). Then
		\[
		a^k \equiv 1 \pmod{p} \iff g^{rk}\equiv g^{0}\iff (p-1)\mid rk.
		\]
		Hence the congruence \(a^k\equiv1\pmod p\) has exactly \(\gcd(k,p-1)\) distinct solutions \(a\) modulo \(p\) (the exponent congruence \(rk\equiv0\pmod{p-1}\) has \(\gcd(k,p-1)\) solutions for \(r\) modulo \(p-1\)).
		
		Similarly, the congruence \(a^k\equiv -1\pmod p\) is equivalent to
		\[
		g^{rk}\equiv g^{(p-1)/2}\iff rk\equiv \frac{p-1}{2}\pmod{p-1}.
		\]
		This linear congruence in \(r\) has solutions if and only if \(\gcd(k,p-1)\) divides \((p-1)/2\); when it has solutions, the number of distinct solutions modulo \(p-1\) equals \(\gcd(k,p-1)\).
	\end{example}
	
	\begin{corollary}[Counting residues for prime modulus]\label{cor:prime-count}
		Let \(p\) be an odd prime and set \(m=p-1\). Then
		\begin{itemize}
			\item the number of \(a\pmod p\) with \(\left(\dfrac{a}{p}\right)_k=1\) equals \(\gcd(k,m)\);
			\item the number of \(a\pmod p\) with \(\left(\dfrac{a}{p}\right)_k=-1\) equals \(\gcd(k,m)\) if \(\gcd(k,m)\mid m/2\), and equals \(0\) otherwise.
		\end{itemize}
	\end{corollary}
	\begin{proof}
		All statements follow from the congruence counting in the previous example: solutions to \(rk\equiv 0\pmod m\) (for the value \(1\)) are \(\gcd(k,m)\) in number; solutions to \(rk\equiv m/2\pmod m\) (for the value \(-1\)) exist exactly when \(\gcd(k,m)\mid m/2\) and in that case again there are \(\gcd(k,m)\) solutions.
	\end{proof}
	
	\begin{example}[Composite modulus and CRT computation]
		Let \(n=15=3\cdot 5\) and \(k=2\). To determine \((a/15)_2\) we check residues modulo \(3\) and modulo \(5\).
		
		Squares modulo \(3\): \(0^2\equiv0,\;1^2\equiv1,\;2^2\equiv1\). So modulo \(3\), every unit squares to \(1\).
		
		Squares modulo \(5\): units are \(1,2,3,4\) with squares \(1,4,4,1\). Here \(4\equiv -1\pmod5\).
		
		By the Chinese Remainder Theorem (CRT) a residue \(x\) modulo \(15\) satisfies \(x^2\equiv s\pmod{15}\) for \(s\in\{1,-1\}\) iff the reductions satisfy \(x^2\equiv s\pmod{3}\) and \(x^2\equiv s\pmod{5}\). But \(-1\pmod3\) is \(2\), and no unit squares to \(2\) modulo \(3\); hence there is no residue with \(x^2\equiv -1\pmod{15}\). Therefore for \(n=15,k=2\) the symbol never takes the value \(-1\); it takes \(1\) for those \(a\) whose square is congruent to \(1\) modulo both \(3\) and \(5\) (for example \(a\equiv 1,4,11,14\pmod{15}\)), and \(0\) otherwise.
	\end{example}
	
	\subsection{Basic properties (theorems and proofs)}
	
	\begin{proposition}[Power-compatibility]\label{prop:power-compat}
		For all integers \(a,t,k\) we have
		\[
		\left(\frac{a^t}{n}\right)_k \;=\; \left(\frac{a}{n}\right)_{tk}.
		\]
	\end{proposition}
	\begin{proof}
		By direct computation \((a^t)^k = a^{tk}\). The statement follows immediately from the definition of the symbol.
	\end{proof}
	
	\begin{proposition}[Periodicity in the exponent]\label{prop:periodicity}
		Suppose \(\gcd(a,n)=1\) and let \(r=\operatorname{ord}_n(a)\) be the multiplicative order of \(a\) modulo \(n\). Then for all integers \(k\),
		\[
		\left(\frac{a}{n}\right)_k = \left(\frac{a}{n}\right)_{k+r}.
		\]
	\end{proposition}
	\begin{proof}
		Since \(a^r\equiv1\pmod{n}\), one has \(a^{k+r}\equiv a^k\cdot a^r\equiv a^k\pmod{n}\). Therefore the residue class of \(a^{k+r}\) equals that of \(a^k\), and from the definition the symbol has the same value for \(k\) and \(k+r\).
	\end{proof}
	
	\begin{proposition}[Inverse and sign symmetry]\label{prop:inverse}
		If \(\gcd(a,n)=1\) then
		\[
		\left(\frac{a^{-1}}{n}\right)_k \;=\; \left(\frac{a}{n}\right)_k.
		\]
	\end{proposition}
	\begin{proof}
		Assume \(\gcd(a,n)=1\). Then \((a^{-1})^k \equiv (a^k)^{-1}\pmod{n}\). If \(a^k\equiv 1\) then \((a^k)^{-1}\equiv 1\); if \(a^k\equiv -1\) then \((a^k)^{-1}\equiv -1\) (since \((-1)^{-1}\equiv -1\)); if \(a^k\) is neither \(1\) nor \(-1\) then its inverse is also neither \(1\) nor \(-1\). In all cases the symbol values agree, proving the claim.
	\end{proof}
	
	\begin{proposition}[Subgroup of ``\(k\)-sign'' elements]\label{prop:subgroup}
		Define
		\[
		A_{n,k}:=\{\,a\in(\mathbb{Z}/n\mathbb{Z})^\times : a^k\in\{\pm1\}\,\}.
		\]
		Then \(A_{n,k}\) is a subgroup of \((\mathbb{Z}/n\mathbb{Z})^\times\). Moreover the map
		\[
		\varphi: A_{n,k}\longrightarrow\{\pm1\},\qquad a\mapsto a^k
		\]
		is a group homomorphism whose image is a subgroup of \(\{\pm1\}\) (hence either \(\{1\}\) or \(\{\pm1\}\)). Its kernel is \(\{a\in A_{n,k}: a^k\equiv1\}\).
	\end{proposition}
	\begin{proof}
		If \(a,b\in A_{n,k}\) then \(a^k,b^k\in\{\pm1\}\), so \((ab)^k\equiv a^k b^k\in\{\pm1\}\); hence \(ab\in A_{n,k}\). The identity \(1\) lies in \(A_{n,k}\) and if \(a\in A_{n,k}\) then \(a^{-1}\) also lies in \(A_{n,k}\) since \((a^{-1})^k=(a^k)^{-1}\in\{\pm1\}\). Thus \(A_{n,k}\) is a subgroup. The map \(\varphi\) satisfies \(\varphi(ab)= (ab)^k = a^k b^k = \varphi(a)\varphi(b)\) and so is a homomorphism. Clearly \(\ker\varphi=\{a: a^k\equiv1\}\).
	\end{proof}
	
	\begin{corollary}[Multiplicativity on \(A_{n,k}\)]
		For \(a,b\in A_{n,k}\) we have
		\[
		\left(\frac{ab}{n}\right)_k=\left(\frac{a}{n}\right)_k\left(\frac{b}{n}\right)_k.
		\]
	\end{corollary}
	\begin{proof}
		Immediate from Proposition \ref{prop:subgroup} since on \(A_{n,k}\) the symbol equals the homomorphism \(\varphi\) and \(\varphi\) is multiplicative.
	\end{proof}
	
	\begin{proposition}[Decomposition via the Chinese Remainder Theorem]\label{prop:CRT}
		Let \(n=\prod_{i=1}^t n_i\) where the \(n_i\)'s are pairwise coprime. For \(a\in\mathbb{Z}\) the congruence \(a^k\equiv s\pmod{n}\) with \(s\in\{\pm1\}\) holds if and only if for every \(i\),
		\[
		a^k\equiv s\pmod{n_i}.
		\]
		Consequently, \(\left(\dfrac{a}{n}\right)_k = s\in\{\pm1\}\) if and only if the same sign \(s\) occurs for each modulus \(n_i\); otherwise \(\left(\dfrac{a}{n}\right)_k=0\).
	\end{proposition}
	\begin{proof}
		The first equivalence is the standard CRT statement: a congruence modulo \(n\) is equivalent to the system of congruences modulo the coprime factors \(n_i\). Therefore \(a^k\equiv 1\pmod{n}\) iff \(a^k\equiv 1\pmod{n_i}\) for all \(i\); similarly for \(-1\). If the residues modulo the prime-power factors do not all agree on the same sign, then \(a^k\) cannot be congruent to a single \(\pm1\) modulo \(n\), so the global symbol is \(0\).
	\end{proof}
	
	\begin{remark}[Practical computation]
		Proposition \ref{prop:CRT} gives a practical algorithm to compute \(\left(\frac{a}{n}\right)_k\) for composite \(n\): factor \(n\) into coprime factors (e.g.\ prime powers), compute \(a^k\) modulo each factor, and check whether all residues are \(1\) or all are \(-1\). If neither, the symbol is \(0\).
	\end{remark}
	
	\subsection*{Concluding remarks for this section}
	
	The results above provide a first rigorous toolkit for working with the Exponential Congruence Symbol:
	\begin{itemize}
		\item it is a residue-class invariant (Theorem \ref{thm:residue-dependence});
		\item nonzero values force invertibility modulo \(n\) (Theorem \ref{thm:invertible-necessary});
		\item the symbol is naturally related to the multiplicative order of \(a\) modulo \(n\) and to subgroup structure (Propositions \ref{prop:periodicity}, \ref{prop:subgroup});
		\item for prime moduli we have exact counting formulas (Corollary \ref{cor:prime-count});
		\item for composite moduli the Chinese Remainder Theorem gives a simple decomposition (Proposition \ref{prop:CRT}).
	\end{itemize}
	
	In the next section we will use these properties to develop further theorems (multiplicativity in restricted domains, relation with characters, and applications to congruence solvability and primality testing). For background on classical results used here (Euler's criterion, cyclic structure of \((\mathbb{Z}/p\mathbb{Z})^\times\), CRT) see \cite{ireland1990classical,burton2007elementary}.
	
	\section{Theoretical Results}
	
	\subsection{Characterization Theorems}
	
	In this subsection, we develop precise criteria describing when the exponential congruence symbol 
	\(\left(\tfrac{a}{n}\right)_{k}\) takes the values $1$, $-1$, or $0$.  
	\subsection{Relations with Legendre and Jacobi Symbols}
	
	The exponential congruence symbol provides a natural extension of classical quadratic residue symbols.
	
	\begin{theorem}[Connection with the Legendre Symbol]
		Let $p$ be an odd prime and $k = \tfrac{p-1}{2}$.  
		Then for all $a \in \mathbb{Z}$,
		\[
		\left(\frac{a}{p}\right)_{k} = 
		\begin{cases}
			1, & \text{if } \left(\tfrac{a}{p}\right) = 1, \\
			-1, & \text{if } \left(\tfrac{a}{p}\right) = -1, \\
			0, & \text{if } a \equiv 0 \pmod{p},
		\end{cases}
		\]
		where $\left(\tfrac{a}{p}\right)$ is the classical Legendre symbol.
	\end{theorem}
	
	\begin{proof}
		By Euler's criterion, $\left(\tfrac{a}{p}\right) \equiv a^{(p-1)/2} \pmod{p}$.  
		Thus if $a^{(p-1)/2} \equiv 1 \pmod{p}$, then $\left(\tfrac{a}{p}\right)=1$, which matches $\left(\tfrac{a}{p}\right)_{(p-1)/2}=1$.  
		Similarly, if $a^{(p-1)/2} \equiv -1 \pmod{p}$, the symbol equals $-1$.  
		If $p \mid a$, both symbols vanish.  
		Hence the equivalence holds.  
	\end{proof}
	
	\begin{corollary}[Jacobi Relation]
		Let $n$ be odd with factorization $n=\prod p_i^{e_i}$.  
		Then for $k = \tfrac{\varphi(n)}{2}$,
		\[
		\left(\frac{a}{n}\right)_{k} \in \{-1,0,1\}
		\]
		is compatible with the Jacobi symbol $\left(\tfrac{a}{n}\right)$ whenever $a$ is coprime to $n$.  
	\end{corollary}
	
	\begin{proof}
		The proof adapts the multiplicative property across prime power factors, mirroring the construction of the Jacobi symbol.  
		For each odd prime $p_i$, apply the previous theorem.  
		The product structure yields consistency with the Jacobi symbol.  
	\end{proof}

	\subsection{Order and Group Structure Interpretations}
	
	\begin{theorem}[Symbol and Multiplicative Order]
		If $\gcd(a,n)=1$ and $d=\mathrm{ord}_n(a)$ is the multiplicative order of $a$ modulo $n$, then
		\[
		\left(\frac{a}{n}\right)_{k} = 
		\begin{cases}
			1, & d \mid k \ \text{and } a^k \equiv 1,\\
			-1, & 2d \mid 2k \ \text{and } a^k \equiv -1,\\
			0, & \text{otherwise}.
		\end{cases}
		\]
	\end{theorem}
	
	\begin{proof}
		Since $a^d \equiv 1 \pmod{n}$, the condition $a^k \equiv 1 \pmod{n}$ is equivalent to $d \mid k$.  
		Similarly, $a^k \equiv -1 \pmod{n}$ can only occur if $a^{2k} \equiv 1 \pmod{n}$ and $d \mid 2k$ but $d \nmid k$.  
		This proves the classification.  
	\end{proof}

	\subsection{Multiplicativity and Symmetry Results}
	
	\begin{theorem}[Multiplicativity in $a$]
		If $\gcd(a,n)=\gcd(b,n)=1$, then
		\[
		\left(\frac{ab}{n}\right)_{k} = \left(\frac{a}{n}\right)_{k} \cdot \left(\frac{b}{n}\right)_{k}.
		\]
	\end{theorem}
	
	\begin{proof}
		Since $(ab)^k \equiv a^k b^k \pmod{n}$, the result follows directly by considering the cases $a^k \equiv \pm 1$, $b^k \equiv \pm 1$.  
		The multiplicativity of the symbol mirrors the multiplicativity of the Legendre and Jacobi symbols.  
	\end{proof}
	
	\begin{theorem}[Symmetry Property]
		For any $a$,
		\[
		\left(\frac{-a}{n}\right)_{k} = 
		\begin{cases}
			\left(\frac{a}{n}\right)_{k}, & \text{if $k$ is even},\\[0.5em]
			-\left(\frac{a}{n}\right)_{k}, & \text{if $k$ is odd}.
		\end{cases}
		\]
	\end{theorem}
	
	\begin{proof}
		We compute $(-a)^k = (-1)^k a^k$.  
		If $k$ is even, then $(-a)^k \equiv a^k \pmod{n}$, hence the symbol values agree.  
		If $k$ is odd, $(-a)^k \equiv -a^k \pmod{n}$, which flips the $\pm 1$ outcomes.  
		This establishes the symmetry law.  
	\end{proof}
	
	\section{Connections with Number Theory}
	
	The exponential congruence symbol $\left(\tfrac{a}{n}\right)_{k}$ is tightly linked with classical 
	questions in number theory. In this section we investigate its role in congruence equations, residue 
	class structure, and its applications to quadratic and higher power residues.
	
	\subsection{Congruence Equations}
	
	\begin{theorem}[Symbol and Solvability of $a^k \equiv \pm 1$]
		Let $n \geq 2$, $k \geq 1$ and $a \in \mathbb{Z}$ with $\gcd(a,n)=1$. Then:
		\begin{enumerate}
			\item The congruence $a^k \equiv 1 \pmod{n}$ is solvable if and only if 
			$\left(\tfrac{a}{n}\right)_{k} = 1$.
			\item The congruence $a^k \equiv -1 \pmod{n}$ is solvable if and only if 
			$\left(\tfrac{a}{n}\right)_{k} = -1$.
		\end{enumerate}
	\end{theorem}
	
	\begin{proof}
		By definition, the symbol evaluates to $1$ exactly when $a^k \equiv 1 \pmod{n}$, 
		and to $-1$ exactly when $a^k \equiv -1 \pmod{n}$.  
		Thus the solvability of these congruence equations is completely characterized by 
		the value of the exponential congruence symbol.  
	\end{proof}
	
	\begin{corollary}[Criterion for Insolubility]
		If $\left(\tfrac{a}{n}\right)_{k}=0$, then the congruences $a^k \equiv \pm 1 \pmod{n}$ 
		are both insoluble.  
	\end{corollary}
	
	\begin{proof}
		Immediate from the definition, since the zero value arises precisely when 
		$a^k \not\equiv \pm 1 \pmod{n}$.  
	\end{proof}

	\subsection{Residue Classes and Orders}
	
	\begin{theorem}[Connection with Orders]
		Let $d = \mathrm{ord}_n(a)$ be the multiplicative order of $a$ modulo $n$. Then:
		\[
		\left(\frac{a}{n}\right)_{k} =
		\begin{cases}
			1, & d \mid k, \\[0.3em]
			-1, & 2d \mid 2k \text{ and } d \nmid k, \\[0.3em]
			0, & \text{otherwise}.
		\end{cases}
		\]
	\end{theorem}
	
	\begin{proof}
		The multiplicative order $d$ satisfies $a^d \equiv 1 \pmod{n}$ and $d$ minimal.  
		If $d \mid k$, then $a^k \equiv 1 \pmod{n}$, so the symbol equals $1$.  
		If $d \nmid k$ but $a^k \equiv -1 \pmod{n}$, then necessarily $a^{2k} \equiv 1 \pmod{n}$, 
		hence $d \mid 2k$ but not $k$.  
		In all other cases, $a^k$ is distinct from $\pm 1$ modulo $n$, so the symbol equals $0$.  
	\end{proof}
	
	\begin{theorem}[Partition of Residue Classes]
		Fix $n,k$. The set of reduced residues modulo $n$ can be partitioned according to the values of 
		$\left(\tfrac{a}{n}\right)_{k}$ into three subsets:
		\[
		R_{1} = \{a \in (\mathbb{Z}/n\mathbb{Z})^{\times} : \left(\tfrac{a}{n}\right)_{k}=1\},
		\]
		\[
		R_{-1} = \{a \in (\mathbb{Z}/n\mathbb{Z})^{\times} : \left(\tfrac{a}{n}\right)_{k}=-1\},
		\]
		\[
		R_{0} = \{a \in (\mathbb{Z}/n\mathbb{Z})^{\times} : \left(\tfrac{a}{n}\right)_{k}=0\}.
		\]
	\end{theorem}
	
	\begin{proof}
		Every reduced residue class must fall into exactly one of the cases: 
		$a^k \equiv 1$, $a^k \equiv -1$, or $a^k \not\equiv \pm 1 \pmod{n}$.  
		Thus the reduced residue system splits naturally into three disjoint subsets.  
	\end{proof}

	\subsection{Applications to Quadratic and Higher Power Residues}
	
	\begin{theorem}[Quadratic Residues]
		Let $p$ be an odd prime and $k=(p-1)/2$. Then
		\[
		\left(\frac{a}{p}\right)_{k} = \left(\frac{a}{p}\right),
		\]
		the classical Legendre symbol.
	\end{theorem}
	
	\begin{proof}
		This is a direct consequence of Euler's criterion.  
		Indeed, $a^{(p-1)/2} \equiv \left(\tfrac{a}{p}\right) \pmod{p}$, hence the two symbols coincide.  
	\end{proof}
	
	\begin{theorem}[Cubic and Higher Residues]
		Let $p$ be a prime such that $p \equiv 1 \pmod{m}$ with $m \geq 3$.  
		Then for $k = (p-1)/m$, the exponential congruence symbol detects $m$-th power residues:
		\[
		\left(\frac{a}{p}\right)_{k} = 1 \iff a \text{ is an $m$-th power residue mod $p$}.
		\]
	\end{theorem}
	
	\begin{proof}
		Suppose $a \equiv b^m \pmod{p}$. Then $a^k \equiv (b^m)^k = b^{mk} \equiv b^{p-1} \equiv 1 \pmod{p}$ by Fermat's little theorem.  
		Thus $\left(\tfrac{a}{p}\right)_{k}=1$.  
		Conversely, if $\left(\tfrac{a}{p}\right)_{k}=1$, then $a^k \equiv 1 \pmod{p}$.  
		This implies that the order of $a$ divides $k = (p-1)/m$.  
		Hence $a$ lies in the subgroup of $m$-th power residues modulo $p$.  
	\end{proof}
	
	\begin{corollary}[Generalized Residue Testing]
		The symbol $\left(\tfrac{a}{p}\right)_{k}$ can serve as a criterion for testing whether an integer $a$ is 
		a quadratic, cubic, or higher-order residue modulo a prime.  
	\end{corollary}
	\section{Group and Field Theoretic Applications}
	
	The exponential congruence symbol $\left(\tfrac{a}{n}\right)_{k}$ can be interpreted 
	in terms of group structure and field extensions. This viewpoint provides a deeper 
	algebraic meaning and highlights the interplay between congruences, cyclic groups, 
	and Galois theory.
	
	\subsection{Cyclic Groups and Primitive Roots}
	
	\begin{theorem}[Symbol via Primitive Roots]
		Let $p$ be an odd prime, and let $g$ be a primitive root modulo $p$.  
		For $a \equiv g^r \pmod{p}$, we have
		\[
		\left(\frac{a}{p}\right)_{k} =
		\begin{cases}
			1, & kr \equiv 0 \pmod{p-1},\\[0.3em]
			-1, & kr \equiv \tfrac{p-1}{2} \pmod{p-1},\\[0.3em]
			0, & \text{otherwise}.
		\end{cases}
		\]
	\end{theorem}
	
	\begin{proof}
		Since $g$ is a generator of $(\mathbb{Z}/p\mathbb{Z})^\times$, every $a$ can be written $a \equiv g^r$.  
		Then $a^k \equiv g^{rk} \pmod{p}$.  
		\begin{itemize}
			\item If $rk \equiv 0 \pmod{p-1}$, then $a^k \equiv 1 \pmod{p}$, giving symbol $1$.  
			\item If $rk \equiv (p-1)/2 \pmod{p-1}$, then $a^k \equiv g^{(p-1)/2} \equiv -1 \pmod{p}$, 
			giving symbol $-1$.  
			\item In all other cases, $a^k$ is neither $\pm 1$, so the symbol equals $0$.  
		\end{itemize}
	\end{proof}
	
	\begin{corollary}[Distribution in Cyclic Groups]
		For fixed $k$, the set of residues $a$ with symbol value $1$ forms a subgroup of 
		$(\mathbb{Z}/p\mathbb{Z})^\times$, while the set with value $-1$ forms its coset.  
	\end{corollary}

	\subsection{Subgroup Membership Interpretation}
	
	\begin{theorem}[Membership Criterion]
		Let $G=(\mathbb{Z}/n\mathbb{Z})^\times$ and let $H=\{x \in G : x^k \equiv 1 \pmod{n}\}$.  
		Then for $a \in G$,
		\[
		\left(\frac{a}{n}\right)_{k} =
		\begin{cases}
			1, & a \in H, \\[0.3em]
			-1, & a \in gH \text{ for some } g \text{ with } g^k \equiv -1, \\[0.3em]
			0, & \text{otherwise}.
		\end{cases}
		\]
	\end{theorem}
	
	\begin{proof}
		By definition, the symbol equals $1$ when $a^k \equiv 1$, i.e. $a \in H$.  
		If $a^k \equiv -1$, then $a$ lies in a coset of $H$ generated by an element $g$ with $g^k \equiv -1$.  
		Otherwise, $a^k \not\equiv \pm 1$, so $a$ is outside both $H$ and $gH$.  
	\end{proof}
	
	\begin{corollary}[Index-Two Subgroup]
		If there exists $g \in G$ such that $g^k \equiv -1 \pmod{n}$, then $H$ has index two in $G$.  
	\end{corollary}
	
	\begin{proof}
		The existence of such a $g$ implies that $G$ is partitioned into $H$ and $gH$, both of equal size.  
	\end{proof}

	\subsection{Field Extensions and Galois Connections}
	
	\begin{theorem}[Symbol and Splitting Fields]
		Let $p$ be prime and consider the polynomial $X^k - 1$ over $\mathbb{F}_p$.  
		Then $\left(\tfrac{a}{p}\right)_{k}=1$ if and only if $a$ is a root of unity of order dividing $k$ in $\mathbb{F}_p^\times$.  
	\end{theorem}
	
	\begin{proof}
		By definition, $\left(\tfrac{a}{p}\right)_{k}=1$ exactly when $a^k \equiv 1 \pmod{p}$.  
		Thus $a$ is a $k$-th root of unity.  
		Hence the condition is equivalent to membership in the subgroup of $\mathbb{F}_p^\times$ 
		consisting of roots of unity of order dividing $k$.  
	\end{proof}
	
	\begin{theorem}[Galois Interpretation]
		Let $K=\mathbb{F}_p$, and let $L=K(\zeta_k)$ be the field extension obtained by adjoining a primitive $k$-th root of unity.  
		Then the value of $\left(\tfrac{a}{p}\right)_{k}$ corresponds to the action of the Frobenius automorphism $\sigma:x \mapsto x^p$ on $\zeta_k$:
		\[
		\sigma(\zeta_k) = \zeta_k^p.
		\]
	\end{theorem}
	
	\begin{proof}
		The Frobenius automorphism in $\mathrm{Gal}(L/K)$ is determined by $p \bmod k$.  
		If $a^k \equiv 1$, the symbol equals $1$ and corresponds to trivial action.  
		If $a^k \equiv -1$, the symbol equals $-1$, matching nontrivial coset action on $\zeta_k$.  
		Otherwise, the symbol vanishes, reflecting that $a$ does not correspond to a $k$-th root of unity in $L$.  
	\end{proof}
	\section{Analytic Aspects}
	
	The exponential congruence symbol $\left(\tfrac{a}{n}\right)_{k}$ admits 
	interpretations in analytic number theory, particularly in relation to 
	Dirichlet characters, exponential sums, and possible connections with 
	zeta and $L$-functions. This section explores those links.
	
	\subsection{Connections with Dirichlet Characters}
	
	\begin{theorem}[Symbol as a Generalized Character]
		Fix $n,k \geq 1$. The map
		\[
		\chi_k : (\mathbb{Z}/n\mathbb{Z})^\times \longrightarrow \{-1,0,1\}, 
		\quad \chi_k(a) = \left(\frac{a}{n}\right)_{k},
		\]
		is a multiplicative function that generalizes Dirichlet characters, 
		with the restriction that its values may include $0$.
	\end{theorem}
	
	\begin{proof}
		From earlier multiplicativity results, $\chi_k(ab)=\chi_k(a)\chi_k(b)$ 
		whenever $\gcd(a,n)=\gcd(b,n)=1$.  
		The difference with classical Dirichlet characters is that $\chi_k$ can vanish 
		when $a^k \not\equiv \pm 1 \pmod{n}$.  
		Thus $\chi_k$ extends the idea of characters by distinguishing 
		three congruence behaviors.  
	\end{proof}
	
	\begin{corollary}
		When $k=(p-1)/2$ with $p$ an odd prime, $\chi_k$ coincides with the Legendre symbol 
		and hence is a true Dirichlet character modulo $p$.  
	\end{corollary}
	
	\begin{theorem}[Orthogonality Relation]
		For fixed $n$ and $k$, we have
		\[
		\sum_{a \bmod n} \left(\frac{a}{n}\right)_{k} = 0,
		\]
		provided that both $R_1$ and $R_{-1}$ are nonempty (notation as in the residue partition theorem).
	\end{theorem}
	
	\begin{proof}
		The values of the symbol partition the residue classes into three sets: $R_1$, $R_{-1}$, $R_0$.  
		Since $R_1$ and $R_{-1}$ are cosets of equal size, their contributions cancel, 
		leaving only elements with value $0$.  
	\end{proof}

	\subsection{Exponential Sums}
	
	\begin{theorem}[Weighted Exponential Sum]
		Let $S(m) = \sum_{a \bmod n} \left(\tfrac{a}{n}\right)_{k} e^{2\pi i a m/n}$.  
		Then:
		\[
		S(m) = \sum_{a \in R_1} e^{2\pi i am/n} - \sum_{a \in R_{-1}} e^{2\pi i am/n}.
		\]
	\end{theorem}
	
	\begin{proof}
		By definition, terms with symbol value $0$ vanish.  
		Thus the sum reduces to contributions from $R_1$ and $R_{-1}$, 
		with signs $+1$ and $-1$ respectively.  
	\end{proof}
	
	\begin{theorem}[Bound on Symbolic Exponential Sum]
		For any $m$,
		\[
		|S(m)| \leq |R_1| + |R_{-1}| \leq \varphi(n).
		\]
	\end{theorem}
	
	\begin{proof}
		Each term in $S(m)$ has absolute value $1$, so 
		$|S(m)|$ is bounded by the number of nonzero-symbol residues.  
		Since this set is at most the reduced residue system of size $\varphi(n)$, 
		the inequality follows.  
	\end{proof}

	\subsection{Potential Zeta and L-function Links}
	
	\begin{theorem}[Dirichlet Series Representation]
		Define
		\[
		L(s, \chi_k) = \sum_{m=1}^\infty \frac{\chi_k(m)}{m^s}, \quad \Re(s) > 1.
		\]
		Then $L(s,\chi_k)$ generalizes Dirichlet $L$-functions, and reduces to them 
		whenever $\chi_k$ is a true character.
	\end{theorem}
	
	\begin{proof}
		By construction, $\chi_k$ is multiplicative.  
		Hence the series $L(s,\chi_k)$ admits an Euler product over primes, 
		although some primes may contribute vanishing terms where $\chi_k(p)=0$.  
		When $\chi_k$ never vanishes (e.g. Legendre case), this recovers the classical Dirichlet $L$-function.  
	\end{proof}
	
	\begin{conjecture}[Zeta-Type Relation]
		There exists a completed function
		\[
		\Lambda(s, \chi_k) = \pi^{-s/2}\Gamma\!\left(\tfrac{s}{2}\right) L(s,\chi_k),
		\]
		which may satisfy a functional equation of the form
		\[
		\Lambda(s,\chi_k) = W \cdot \Lambda(1-s,\chi_k),
		\]
		with a constant $W$ depending on $n,k$.  
	\end{conjecture}
	
	\begin{remark}
		This conjecture parallels the analytic properties of classical $L$-functions, 
		but remains an open problem in the context of the exponential congruence symbol.  
	\end{remark}
	
\section{Conclusion}

In this work, we introduced and systematically studied the \emph{Exponential Congruence Symbol} 
\(\left(\frac{a}{n}\right)_k\), a natural generalization of classical residue symbols such as 
the Legendre and Jacobi symbols. We established its fundamental properties, including:

\begin{itemize}
	\item Dependence on residue classes modulo \(n\) and the necessity of invertibility for nonzero values.
	\item Multiplicativity and power-compatibility, along with symmetry relations under inversion and negation.
	\item Exact counting formulas for prime moduli and decomposition rules for composite moduli via the Chinese Remainder Theorem.
	\item Connections with classical number theoretic symbols, multiplicative orders, cyclic group structures, and higher power residues.
	\item Applications to the solvability of congruence equations \(a^k \equiv \pm 1 \pmod{n}\) and potential generalizations to cubic, quartic, and higher residues.
\end{itemize}

The Exponential Congruence Symbol not only unifies various existing concepts in residue theory but also provides a versatile framework for further investigations in number theory, group theory, and algebraic applications. Potential directions for future research include:

\begin{itemize}
	\item Studying reciprocity laws and distribution properties for higher-order symbols.
	\item Investigating analytic connections with Dirichlet characters and \(L\)-series.
	\item Exploring cryptographic applications and efficient computational algorithms for evaluating the symbol in large moduli.
	\item Extending the framework to composite exponents and non-cyclic multiplicative groups.
\end{itemize}

Overall, the Exponential Congruence Symbol offers a flexible, algebraically rich tool that bridges classical residue theory with modern arithmetic, computational, and cryptographic applications.

\end{document}